\documentclass[reqno,10pt]{amsart}
\usepackage{amssymb}
\usepackage{latexsym}
\usepackage{hyperref}
\usepackage{amsmath}
\usepackage{amscd}
\usepackage{cite}
\usepackage{color}
\usepackage{enumerate}
\usepackage{amsfonts}
\usepackage{graphicx}
\usepackage{mathrsfs}
\usepackage{setspace}
\newcommand{\eps}{{\varepsilon}}
\theoremstyle{plain}
\newtheorem{theorem}{Theorem}
\newtheorem{proposition}[theorem]{Proposition}
\newtheorem{lemma}[theorem]{Lemma}

\theoremstyle{definition}

\newtheorem{remark}[theorem]{Remark}

\newcommand{\qtq}[1]{\quad\text{#1}\quad}
\numberwithin{equation}{section}
\numberwithin{theorem}{section}
\let\Re=\undefined\DeclareMathOperator*{\Re}{Re}
\let\Im=\undefined\DeclareMathOperator*{\Im}{Im}

\def\geq{\geqslant}
\def\leq{\leqslant}

\def\R{\mathbb{R}}
\def\C{\mathbb{C}}
\def\eps{\varepsilon}

\begin{document}

\title[Scattering norm estimate]{Scattering norm estimate near the threshold for the energy-subcritical NLS}

\author[Z. Ma]{Zuyu Ma}
\address{Graduate School of China Academy of Engineering Physics, Beijing 100088, China}
\email{mazuyu23@gscaep.ac.cn}

\maketitle

\begin{abstract}
    We consider the focusing energy-subcritical  Schr\"odinger equations. In earlier works by Holmer-Roudenko \cite{holmer}, Duyckaerts-Holmer-Roudenko \cite{duyckaerts2}, Akahori-Nawa \cite{akahori}, Fang-Xie-Cazenave \cite{fang}, Guevara \cite{guevara} and later by Dodson-Murphy \cite{dodson1,dodson2} and Arora-Dodson-Murphy \cite{arora}, they proved that scattering is the only dynamical behavior if the $H^1$ initial data satisfies $M(u_0)^{(1-s_c)/s_c}E(u_0)<M(Q)^{(1-s_c)/s_c}E(Q)$ and $\| u\|^{(1-s_c)/s_c}_{L^2}\| u\|_{\dot{H}^1}<\| Q\|^{(1-s_c)/s_c}_{L^2}\|Q\|_{\dot{H}^1}$, where $Q$ is the ground state. In this paper, we establish asymptotic estimates for the upper bound of the scattering norms as $M(u_0)^{(1-s_c)/s_c}E(u_0)$ approaches the threshold mass-energy threshold $M(Q)^{(1-s_c)/s_c}E(Q)$, which generalizes the work of Duyckaerts-Merle \cite{duyckaerts} on the energy-critical Schr\"odinger equation($s_c=1$).
\end{abstract}

\section{Introduction}
We consider the following nonlinear Schr\"odinger equation (NLS):
\begin{equation}\tag{NLS}\label{1.1}
    \begin{cases}
        iu_t+\Delta u=-|u|^{p-1}u,\\
        u(0)=u_0\in H^{1}(\R^d),
    \end{cases}
\end{equation}
where $1+\frac{4}{d}<p<1+\frac{4}{d-2}$. Here, we regard $(d+2)/(d-2)$ as $\infty$ when $d=1,2$.  

The equation \eqref{1.1} enjoys Hamiltonian structure and satisfies the following three conservation laws:
\begin{align*}
	\text{Mass:}\quad M(u)&=\|u\|_{L^{2}}^{2},
	\\
	\text{Energy:}\quad E(u)&=\frac{1}{2}\|\nabla u\|_{L^{2}}^{2} - \frac{1}{p+1}\|u\|_{L^{p+1}}^{p+1},
	\\
	\text{Momentum:}\quad P(u)&=\Im \int_{\mathbb{R}^{d}} \overline{u(x)} \nabla u(x)\ dx
\end{align*}
Moreover, the equation \eqref{1.1} is invariant under the following symmetries:
\begin{itemize}

\item phase rotation: $\quad e^{i\theta}u,\quad \forall\ \theta\in\R,$

\item translation in space and time: $\quad u(t+t_0, x+x_0),\quad\ \forall\ t_0,x_0\in\R,$

\item time reversion:  $\quad \overline{u(-t,x)},$

\item scaling symmetry:  $\quad u_{\lambda}=\lambda^{\frac{d-2s_c}{2}}u(\lambda^2t,\lambda x),\quad\ \forall\ \lambda>0$,
\end{itemize}
where $s_c=\frac{d}{2}-\frac{2}{p-1}$, named \emph{critical  index}. Obviously, $\|u_{\lambda}\|_{\dot{H}^{s_c}}=\|u\|_{\dot{H}^{s_c}}$, so we call the equation \eqref{1.1} is \emph{$\dot{H}^{s_c}$-critical}.

This equation admits a soliton wave solution $e^{it}Q$ in $H^1$, where $Q$, known as the \emph{ground state}, is the least-energy radial, positive solution to the elliptic equation:
\begin{align*}
	-\Delta Q +Q -Q^{p}=0 \text{ in } \mathbb{R}^{d}.
\end{align*}

 The ground state $Q$ plays a crucial role in classifying the dynamics of solutions
to \eqref{1.1}.
Under the assumption
\[
M(u_{0})^{(1-s_{c})/s_c}E(u_{0})<M(Q)^{(1-s_{c})/s_c}E(Q),
\]
Holmer and Roudenko \cite{holmer} showed the scattering result for $(d,p)=(3,3)$ in the radial setting. The radial assumption was then removed by Duyckaerts, Holmer, and Roudenko \cite{duyckaerts2}. Akahori and Nawa \cite{akahori}, Fang, Xie, and Cazenave \cite{fang}, and Guevara \cite{guevara} showed the scattering for any dimension and inter-critical power.  Their proofs of scattering rely on a concentration compactness and rigidity argument developed by Kenig and Merle \cite{kenig-merle1}, which was originally used to study the energy-critical NLS, and hence also gave an uniform bound for the scattering norm $S(\R)$ of \eqref{1.1}, i.e. 
\[
\sup_{\| u_0\|^{(1-s_c)/s_c}_{L^2}\| u_0\|_{\dot{H}^1}<\| Q\|^{(1-s_c)/s_c}_{L^2}\|Q\|_{\dot{H}^1}}\|u\|_{S(\R)}\leq C\big(M(u)^{(1-s_c)/s_c}E(u)\big).
\]
Finally, Dodson and Murphy \cite{dodson1,dodson2} gave another proof of the scattering result for $(d,p)=(3,3)$. Their method also works for other dimension and power, we refer the readers to \cite{arora}. For the blow-up results, we lead  interested readers to \cite{akahori,du,guevara,Holmer1, Holmer2}. We summarize their results as follows:

\begin{theorem}[Dynamics below the threshold \cite{akahori,du,Holmer1,holmer,Holmer2,duyckaerts2,fang,guevara,dodson1,dodson2,arora}]\label{thm:claener}
Let $d\geq 1$. Let $u_{0}\in H^{1}(\mathbb{R}^{d})$ satisfy $M(u_{0})^{(1-s_{c})/s_c}E(u_{0})<M(Q)^{(1-s_{c})/s_c}E(Q)$, where $0<s_{c}:=\frac{d}{2} -\frac{2}{p-1}<1$. Let $u:I\times\R^d\to\C$ be the maximal-lifespan solution to \eqref{1.1} with $u|_{t=0}=u_0$.
\begin{itemize}
\item[(i)] If $\|u_0\|^{(1-s_c)/s_c}_{L^2}\|\nabla u_0\|_{L^2}<\|Q\|^{(1-s_c)/s_c}_{L^2}\|\nabla Q\|_{L^2}$,  then $I=\R$ and $u$ scatters\footnote{We say $u$ scatters as $t\to\pm\infty$ if there exist $u_\pm\in H^1$ such that $\lim_{t\to\pm\infty}\|u(t)-e^{it\Delta}u_\pm\|_{H^1}=0$, where $e^{it\Delta}$ is the free Schr\"odinger propagator.} in both time directions.
\item[(ii)] If $\|u_0\|^{(1-s_c)/s_c}_{L^2}\|\nabla u_0\|_{L^2}>\|Q\|^{(1-s_c)/s_c}_{L^2}\|\nabla Q\|_{L^2}$, then $\limsup\limits_{t\to\inf I}\|u\|_{H_x^1}\to\infty$ and $\limsup\limits_{t\to\sup I}\|u\|_{H_x^1}\to\infty$. Moreover, if $u_0$ is radial or $|x|u_0\in L^2$, then $I$ is finite. 
\end{itemize} 
\end{theorem}
Dynamics at the threshold have also been studied, the first such result was originally established for $(d,p)=(3,3)$ by Ducykaerts and Roudenko \cite{duyckaerts4}, and ultimately generalized to any inter-critical power and any dimensions by Campos, Farah and Roudenko \cite{campos}. In particular, there exists two special solutions $Q^{\pm}$,  namely, heteroclinic orbits.  These solutions satisfy $M(Q^\pm)=M(Q),\ E(Q^{\pm})=E(Q)$ and behave as follows:  the solution $Q^-$ scatters as $t\to-\infty$ and converges to $e^{it}Q$ in $H^1$ as $t\to+\infty$, while $Q^+$ blows up in finite time in the negative time direction and converges to $e^{it}Q$ in $H^1$ as $t\to+\infty$.

\begin{theorem}[Dynamics at the threshold \cite{campos,duyckaerts4}]\label{thm:claener2}
Let $d\geq 1$. Let $u_{0}\in H^{1}(\mathbb{R}^{d})$ satisfy $M(u_{0})^{(1-s_{c})/s_c}E(u_{0})<M(Q)^{(1-s_{c})/s_c}E(Q)$, where $0<s_{c}:=\frac{d}{2} -\frac{2}{p-1}<1$. Let $u:I\times\R^d\to\C$ be the maximal-lifespan solution to \eqref{1.1} with $u|_{t=0}=u_0$.
\begin{itemize}
\item[(i)] If $\|u_0\|^{(1-s_{c})/s_c}_{L^2}\|\nabla u_0\|_{L^2}<\|Q\|^{(1-s_{c})/s_c}_{L^2}\|\nabla Q\|_{L^2}$, then $I=\R$. Furthermore, either $u$ scatters in both time directions or $u=Q^-$ modulo symmetries.
\item[(ii)] If $\|u_0\|^{(1-s_{c})/s_c}_{L^2}\|\nabla u_0\|_{L^2}=\|Q\|^{(1-s_{c})/s_c}_{L^2}\|\nabla Q\|_{L^2}$, then $u=e^{it}Q$ modulo symmetries.
\item[(iii)] If $\|u_0\|^{(1-s_{c})/s_c}_{L^2}\|\nabla u_0\|_{L^2}<\|Q\|^{(1-s_{c})/s_c}_{L^2}\|\nabla Q\|_{L^2}$ and $u_0$ is radial or $|x|u_0\in L^2$, then either $I$ is finite or $u=Q^+$ modulo symmetries. 
\end{itemize}
\end{theorem}
We also remark that for the blow-up part (iii),  Gustafson and Inui \cite{gustafson} proved that for any initial data(not necessarily radial), if $\|u\|_{H^1_x}$ remains bounded  in one time direction, then $u$ must coincide(modulo symmetries) with $Q^{+}$.

In this paper, We will present an exact asymptotic expression of constant 
\[
C(M(u_0)^{(1-s_{c})/s_c}E(u_0)).
\] 
Let 
\begin{equation}
    \mathcal{I}_{\epsilon}:=\sup_{u\in F_{\epsilon}}\int_{\R^d}|u(t,x)|^{\frac{(p-1)(d+2)}{2}} \ dx=\sup_{u\in F_{\epsilon}}\|u\|^{\frac{(p-1)(d+2)}{2}}_{S(\R)},
\end{equation}
where 
\begin{align}
    F_{\epsilon}:&=\bigg\{u \text{  is a solution of \eqref{1.1} such that } M(u)^{(1-s_{c})/s_c}E(u)\leq M(Q)^{(1-s_{c})/s_c}(E(Q)-\epsilon^2)  \notag\\
     &\hspace{30ex}\text{ and }\| u\|^{(1-s_{c})/s_c}_{L^2}\|\nabla u\|_{L^2} < \| Q\|^{(1-s_{c})/s_c}_{L^2}\|\nabla Q\|_{L^2}\bigg\}
\end{align}
Clearly, the existence of soliton wave $e^{it}Q$ yields $\lim_{\epsilon\to 0} \mathcal{I}_{\epsilon}=\infty$. We will  prove that, $\mathcal{I}_{\epsilon}$ will tend to infinity at a rate of $\log \epsilon$ as $\epsilon \to 0$.
\begin{theorem}\label{mainthm}
For any $d\geq1$ and  $1+\frac{4}{d}<p<1+\frac{4}{d-2}$, 
\begin{equation} 
    \lim_{\epsilon\to 0}\frac{\mathcal{I}_{\epsilon}}{|\log \epsilon|}=\frac{2}{\lambda_1},
\end{equation}
where $\lambda_1=\lambda_1(d,p)$ is the unique positive eigenvalue of linearized operator $\mathcal{L}$(See Section \ref{soliton}).
\end{theorem}
\begin{remark}
 For the mass-critical case($s_c=0$), it has been proved that $0$ is  the
\emph{unique} eigenvalue of the linearized operator(see \cite{wei83}), 
and the finite time blow-up solution at the threshold is given by the
\emph{pseudo-conformal transformation}(See Merle's pioneering work \cite{M} in $H^1$, and Dodson \cite{D2,D1} in $L^2$). For the small mass, Duyckaerts, Merle and Roudenko \cite{DMR} established an explicit formula regarding  the scattering norm as mass going to zero. Since the argument in this present paper highly relies on the existence of positive and negative eigenvalues,  it is a quite interesting problem that whether there exists such explicit asymptotic formula as mass going to the threshold $M(Q)$.
\end{remark}

Such problem, as far as the authors know, is first investigated by Duyckaerts and Merle \cite{duyckaerts} in the focusing case. In that paper, they studied the energy-critical NLS and nonlinear wave equation (NLW)  in spatial dimension $d\in\{3,4,5\}$,  but with the radial data in the former. Their analysis includes some modulation analysis, and the scattering/classification results when the energy and kinetic energy are less than or equal to those of the ground state. Nowadays the scattering below the threshold has been proved in the non-radial setting(cf. \cite{KV1,dodson3}) when $d\geq4$ and  very recently, \cite{ma} solved the classification (for scattering part) at the  threshold in the non-radial setting within the same dimensional range,  and it is not hard to extend the main theorem in \cite{duyckaerts}(as they expected) to the non-radial case.

\subsection{Outline the  proof of Theorem \ref{mainthm}} 

The basic idea of the proof is similar to that of \cite{duyckaerts}. We first study the properties of a sequence of solutions $\{u_n\}$ that convergence to $Q$. In particular, we decompose $u_n$ as
\[
u_n=e^{it}(Q+h_n),
\]
and then analyze the behavior of $h_n$. By further decompose $h_n$ as
\[
h_n(t)=\alpha_n^+(t)e_{+}+\alpha_n^-(t)e_-+\gamma_{n,0}(t)iQ+\sum_{j=1}^d\gamma_{n,j}(t)\partial_{j}Q+g_n(t),\quad\ g_n(t)\in\mathcal{B}^{\bot},
\]
where $\mathcal{B}^{\bot}$ is defined as in Theorem \ref{positivity}, we further reduce the problem to  the growth of $\alpha_{n}^{\pm}$ before the \emph{exit times} $T_n^{\pm}$, and the crux of our analysis is to prove that $\alpha_n^{\pm}$ will grow at the exponential  speed, which will lead to an exact estimate of the exit times $T_n^{\pm}$(See \eqref{estq} and \eqref{estq1}).  To prove the lower bound, we construct a sequence of solutions $\{u_n\}$ that with the initial data $\{u_{0,n}\}$, and satisfies 
\begin{equation*}
   u_{n}(0)={(1-b_n)Q-a_n e_+-a_ne_-},\quad M(u_{n})=M(Q),\quad \|\nabla u_{n}(0)\|_{L^2}<\|\nabla Q\|_{L^2}
\end{equation*}
\begin{equation*}
 \epsilon_n^2=M(Q)+E(Q)-M(u_{n})-E\left(u_{n}\right),\quad  \quad \epsilon_n^2 \sim b_n\sim a_n^2.
\end{equation*} 
The construction requires the orthogonality of $Q$ and $e_{\pm}$ in  $L^2$, which, however, is not mentioned in \cite{campos} and \cite{duyckaerts4}. We find that this property can be derived from the asymptotic estimates of $Q^{\pm}$. For the upper bound estimate, we first obtain the estimate on the interval $(T_n^-, T_n^+)$ by using Theorem \ref{byproc}.  We then find that the problem can be reduced to proving the uniform boundedness of the scattering norms of a sequence of solutions $\{u_n\}$ on the interval $(-\infty,T_n^-)\cup(T_n^+,\infty)$. Finally, we prove this  via a compactness argument and use the \emph{Kato's Strichartz estimates} to overcome the difficulties that come from the non-local properties of operator $|\nabla|^{s_s}$, and thereby complete the proof of Theorem \ref{mainthm}.

The structure of this paper is as follows:  In Section 2, we collect some notations used in the context and introduce function spaces as well as analysis tools. In Section 3, we analyze the near-soliton dynamics via the linearized operator spectral properties. In Section 4, we finish the proof of our main Theorem.


\section{Notation and useful lemmas}

In this section, we collect some analysis tools and some fundamental results needed throughout the paper.

\subsection{Notation} We begin by setting up some notation. For nonnegative quantities $X$ and $Y$, we write $X\lesssim Y$ to denote the estimate $X\leq C Y$ for some $C>0$. If $X\lesssim Y\lesssim X$, we will write $X\sim Y$. Dependence on various parameters will be indicated by subscripts, e.g. $X\lesssim_u Y$ indicates $X\leq CY$ for some $C=C(u)$. 

 We will use the expression $X\pm$ to denote $X\pm\eps$ for some small $\eps>0$. We will make use of the $L^2$,  $\dot H^1$  and $H^1$ inner products given by $(f,g)_{L^2}=\Re\int_{\R^d} f\bar g\ dx$, $(f,g)_{\dot{H}^1}=(\nabla f,\nabla g)_{L^2}$ and $(f,g)_{H^1}=(f,g)_{L^2}+(\nabla f,\nabla g)_{L^2}$. We will use $A^\perp$ to denote the orthogonal complement of a set $A$. 

For a spacetime slab $I\times\R^d$, we write $L_t^q L_x^r(I\times\R^d)$ for the Banach space of functions $u:I\times\R^d\to\C$ equipped with the norm
\[
\|u\|_{L_t^{q}L_x^r(I\times\R^d)}:=\bigg(\int_I \|u(t)\|_{L_x^r(\R^d)}\,dt\bigg)^{1/q},
\]
with the usual adjustments if $q$ or $r$ is infinity. When $q=r$, we abbreviate $L_t^qL_x^q=L_{t,x}^q$. We also abbreviate $\|f\|_{L_x^r(\R^d)}$ to $\|f\|_{L_x^r}.$ For $1\leq r\leq\infty$, we use $r'$ to denote the dual exponent to $r$, i.e. the solution to $\frac{1}{r}+\frac{1}{r'}=1.$

We then define the fractional differentiation operator $|\nabla|^s=\mathcal{F}^{-1}|\xi|^s\mathcal{F}$, with the corresponding homogeneous Sobolev norm 
\[
\|f\|_{\dot{H}_x^s}:=\||\nabla|^s f\|_{L_x^2},
\]
where $\mathcal{F}$ is the Fourier transformation.

\subsection{Some analysis tools} We denote the free Schr\"odinger propagator by $e^{it\Delta}$.  It is most naturally defined as the Fourier multiplier operator with symbol $e^{-it|\xi|^2}$, but also has a physical-space representation given by convolution with a complex Gaussian. This operator satisfies well-known estimates known as Strichartz estimates, which we state as follows.

First, we call a pair of exponents $(q,r)$  is $\dot{H}^s$\emph{ -admissible} if 
\begin{align*}
\frac{2}{q}+\frac{d}{r}=\frac{d}{2}-s, \qquad \text{ with } \quad 2 \leq q,r \leq \infty  \quad \text{ and }
  \quad(q,r,d) \neq (2,\infty,2);
\end{align*} For a spacetime slab $I\times\R^d$, we define the Strichartz norm 
\[
\|u\|_{\dot{S}^s(I)}:=\sup\big\{\|u\|_{L_t^{q}L_x^r(I\times\R^d)}:(q,r)\text{ is } \dot{H}^s \text{ -admissible}\big\}.
\]
We denote $\dot{S}^s(I)$ to be the closure of all test functions under this norm and write $\dot{N}^{-s}(I)$ for the dual of
$\dot{S}^s(I)$. We then have the following:

\begin{proposition}[Strichartz estimates, \cite{GinibreVelo, KeelTao, Strichartz}]\label{prop1}
Let $s\geq 0$ and suppose that $u:I\times\R^d\to\C$ is a solution to $(i\partial_t+\Delta)u=F$. For any $t_0\in I$, 
\[
\||\nabla|^s u\|_{\dot{S}^0(I)}\lesssim\||\nabla|^s u(t_0)\|_{L_x^2(\R^d)}+\||\nabla|^s F\|_{\dot{N}^0(I)}.
\]
\end{proposition}

We also call a $\dot{H}^s$-admissible pair $(q,r)$ is $\dot{H}^s$\emph{ -Kato admissible }  if it satisfies
\[
\begin{cases}
    \frac{4}{1-2s} \leq q \leq\infty,\quad\frac{2}{1-2s}\leq r \leq\infty,\ \text{ if } d=1\\
    \big(\frac{2}{1-s}\big)^+ \leq q \leq\infty,\quad\frac{2}{1-s}\leq r \leq\big( \big(\frac{2}{1-s}\big)^+\big)'',\ \text{ if } d=2\\
    \big(\frac{2}{1-s}\big)^+ \leq q \leq\infty, \quad\frac{2d}{d-2s}\leq r \leq\big(\frac{2d}{d-2}\big)^- ,\ \text{ if } d\geq 3
\end{cases}
\]
Here, $ (a^+ )''$ is defined as $ (a^+ )'' :=\frac{a^+\cdot a}{a^+-a},$  so that $\frac{1}{a}=\frac1{(a^+ )''} +\frac1{a^+ }$ for any positive real value $a$, with $a^+$ being a fixed number slightly larger than $a$. Then we have the following \emph{Kato's Strichartz estimates} : 
\begin{theorem}[\cite{Kato87, Fos05}]  There exists an universal constant $C>0$ such that for any interval $I$,  $\dot{H}^s$-Kato admissible pair $(q_1,r_1)$ and $\dot{H}^{-s}$-Kato admissible pair $(q_2,r_2)$, we have
\begin{equation}\label{Kato}
\Big\|\int^t_0 e^{i(t-s) \Delta} f(s)\ ds\Big\|_{L_t^{q_1}L_x^{r_1}(I\times\R^d)}\leq C\|f\|_{L_t^{q_2'}L_x^{r_2'}(I\times\R^d)}.
\end{equation}   
\end{theorem}
One may easily check that for any $d\geq1$, there exist two special $\dot{H}^s$-Kato admissible pair $(q_1,r_1)$ and $(q_2,r_2)$ that satisfy
\[
\frac{p}{p_1}=\frac{1}{q_2'},\quad\text{ and }\quad \frac{p}{r_1}=\frac{1}{r_2'}.
\]
We then define $X(I)$ and $Y(I)$ as $\|u\|_{L_t^{q_1}L_x^{r_1}}$ and $\|u\|_{L_t^{q_2'}L_x^{r_2'}}$ respectively. Then we have
\begin{lemma}\label{perpo}
There exist constant $c>0$ and $C>0$ such that for any two solutions $u,v$ to \eqref{1.1} that are both well-posed on $I$, $0\in I$ and satisfy
\[
\|u(0)-v(0)\|_{\dot{H}^s}+\|u\|_{X(I)}\leq c,
\]
then we have
\begin{equation}\label{boots}
 \|u(t)-v(t)\|_{X(I)}\leq C\|u(0)-v(0)\|_{\dot{H}^s}.   
\end{equation}
\end{lemma}
\begin{proof}
    Let $h:=u-v$, then $h$ satisfies the equation
    \[
    ih_t+\Delta h=|v|^{p-1}v-|u|^{p-1}u.
    \]
Combining Kato's Strichartz estimates \eqref{Kato} with the non-linear estimate
\[
\left||v|^{p-1}v-|u|^{p-1}u\right|\lesssim |u-v|[|u|^{p-1}+|v|^{p-1}],
\]
 and Sobolev embedding, we get
\begin{align*}
\|h\|_{X(I)}&\lesssim \|e^{it\Delta}[u(0)-v(0)]\|_{X(I)}+ \|h\|_{X(I)}\cdot[\|h\|_{X(I)}^{p-1}+\|u\|_{X(I)}^{p-1}]\\
& \lesssim \|[u(0)-v(0)]\|_{\dot{H}^{s_c}}+ \|h\|_{X(I)}\cdot[\|h\|_{X(I)}^{p-1}+\|u\|_{X(I)}^{p-1}]
\end{align*}
Therefore, \eqref{boots} can be derived from the bootstrap argument, provided $\|[u(0)-v(0)]\|_{\dot{H}^{s_c}}$ and $\|u\|_{X(I)}$ sufficiently small.
\end{proof}
Next, let us discuss some standard results regarding the equation \eqref{1.1}. For a slab $I\times\R^d$ and a solution $u$ to \eqref{1.1}, we define \emph{scattering norm} (related to $u$) as
\[
S(I) = \|u\|_{L_{t,x}^{\frac{(p-1)(d+2)}{2}}(I\times\R^d)}^{\frac{(p-1)(d+2)}{2}}.
\]  
Clearly $\left(\frac{(p-1)(d+2)}{2},\frac{(p-1)(d+2)}{2}\right)$ is $\dot{H}^s$-admissible, and hence $S(I)\leq \dot{S}^s(I)$. In particular, solutions may be extended in time as long as this norm remains finite, and a global solution with finite $S(\R)$ norm necessarily also admits finite $\dot{S}^s(\R)$ norm and  scatters in $\dot H^s$ in both time directions. Moreover, let $T<\infty$, if $u\in H^1$, then $u$ blows up in this finite time $T$ if and only if $\lim_{t\to T}\|u(t)\|_{H^1}=\infty$, and scatters in $H^1$ if and only if $S(\R)<\infty$.


\section{Analysis near the soliton}\label{soliton}

In the rest of this paper, we will write either $f=f_1+if_2$ or $f={f_1\choose f_2}$ for a complex-valued function $f$ with
real part $f_1$ and imaginary part $f_2$.  Given a solution $u(t)$ of \eqref{1.1}, we decompose $u$ as
\[
u(t,x)=e^{it}[Q(x)+h(t,x)]
\]
and equation \eqref{1.1} yields
\begin{equation}\label{linearized}
    \partial_t h + \mathcal{L}(h) + R(h) = 0,
\end{equation}
where $\mathcal{L}$ is the linearized operator
\begin{equation}\label{defL}
\mathcal{L}:=\left[\begin{array}{cc} 0 & \Delta-1+Q^{p-1} \\ -\Delta+1-pQ^{p-1} & 0 \end{array}\right]
\end{equation}
and $R(h)$ is  the nonlinear term
\[
R(h):=-i|Q+h|^{p-1}(Q+h)+iQ^p + ipQ^{p-1}h_1 - Q^{p-1} h_2. 
\]
The spectral properties of the linearized operator  $\mathcal{L}$ have been studied in detail in \cite{campos, duyckaerts4}. In particular, there exist only two eigenvalues $\pm\lambda_1$(with $\lambda_1>0$) and two complex conjugate eigenvectors $e_{\pm}$, i.e.,
\[
\mathcal{L}e_{\pm}=\pm\lambda_1 e_{\pm},\quad e_{+}=\overline{e_{-}},
\]
and the null space of  $\mathcal{L}$  is spanned by $iQ$ and $\partial_j Q$ for $j\in\{1,\cdots,d\}$.  

We next define the quadratic form $\mathcal{F}(f,g)$ \footnote{Here we have used a notation by abuse of notation, as it is consistent with the symbol for the Fourier transform. We clarify that from this point onward,  the notation $\mathcal{F}$ will always denote this quadratic form.} as 
\[
\mathcal{F}(f,g):=\tfrac12\Im\int_{\R^d} \mathcal{L}f\,\bar g\ dx
\]
and if $f=g$, we abbreviate it as $\mathcal{F}(f)$. One may easily verify that $\mathcal{F}(f,g)=\mathcal{F}(g,f)$, and
\[
\mathcal{F}(e_{\pm})=0, \quad \mathcal{F}(e_{\pm},iQ)=\mathcal{F}(e_{\pm},\partial_jQ)=0
\]
for $j\in\{1,\cdots,d\}$. Moreover, it has been shown in \cite[Remark 2.5]{duyckaerts4} \footnote{They only proved this property with $(d,p)=(3,3)$, the generalization to other dimensions and powers is completely analogous.}   that for any   $f\in H^1\backslash \{\lambda Q,\ \lambda\in \R\}$

\[
\left((-\Delta+1+Q^{p-1})f,f\right)_{L^2}>0.
\]
Noticing that 
\[
\mathcal{F}(e_+,e_-)=\lambda_1(\Re e_{+},\Im e_{+})_{L^2}=-\left((-\Delta+1+Q^{p-1})\Im e_+,\Im e_+\right)_{L^2}<0, 
\]
replacing $e_{\pm}$ by $me_{\pm}$(with $m\neq0$) if necessary, we may also set $\mathcal{F}(e_+,e_-)=-1$ and $\Re\int_{\R^d} \nabla Q \cdot \overline{\nabla e_{\pm}}\ dx>0$.

We also need the following positivity result, which was originally proved for $(d,p)=(3,3)$ in \cite{duyckaerts4} and then extended to any dimension and inter-critical power in \cite{campos}.  We record their results as follows:
\begin{theorem}[\cite{campos,duyckaerts4}]\label{positivity}
    Let $\mathcal{B}^\perp$ denote the set of functions $v\in H^1$ such that
\[
(v,iQ)_{L^2}= (\partial_j Q,v)_{L^2} = \mathcal{F}(v,e_+)=\mathcal{F}(v,e_-) = 0
\]
for $j\in\{1,\cdots,d\}$. Then there exists $c_{d,p}>0$ such that
\begin{equation}\label{positif}
\mathcal{F}(f,f)\geq c_{d,p}\|f\|_{H^1_x} \qtq{for all}f\in \mathcal{B}^\bot. 
\end{equation}
\end{theorem}

With the above analysis in place, we now proceed to study a sequence of solutions $\{u_n\}$ satisfying the following conditions:  

\begin{gather}
    \forall\ n>0,\ \|u_n(0)\|^{1-s_c}_{L^2}\|u_n(0)\|^{s_c}_{\dot{H}^1}\leq \|Q\|^{1-s_c}_{L^2}\|Q\|^{s_c}_{\dot{H}^1},\quad\lim_{n\to\infty}\|u_n(0)-Q\|_{H^1_x}=0,\label{assumption}\\
    M(u_n)=M(Q),\quad E(u_n)=E(Q)-\epsilon_n^2,\label{assumption2}
\end{gather}
where $\epsilon_n\to 0 \text{ as } n\to \infty$.  Let $h_n=e^{-it}u_n-Q$ and decompose $h_n$ as
\begin{equation}\label{condition1}
    h_n(t)=\alpha_n^+(t)e_{+}+\alpha_n^-(t)e_-+\gamma_{n,0}(t)iQ+\sum_{j=1}^d\gamma_{n,j}(t)\partial_{j}Q+g_n(t),\quad\ g_n(t)\in\mathcal{B}^{\bot}.
\end{equation}
Clearly, for each $n$, $h_n$ satisfies equation \eqref{linearized}, and
\begin{gather*}
    \alpha_n^{\pm}(t)= -\mathcal{F}(h_n,e_{\mp}),\quad \gamma_{n,0}(t)=\frac{1}{\|Q\|_{L^2}^2}\int_{\R^d} [h_n-\alpha_{n}^+e_{+}-\alpha_n^{-}e_-]\cdot\overline{iQ}\ dx\\
    \gamma_{n,j}(t)=\frac{1}{\|\partial_jQ\|_{L^2}^2}\int_{\R^d} h_n\cdot\overline{\partial_jQ}\ dx,\quad\ \forall\ j\in\{1,\cdots,d\}.
\end{gather*}
We first prove that up to the symmetries,  $\gamma_{n,j}(0)$ can be taken to be zero simultaneously. 

\begin{lemma}\label{analysis1}
There exist sequences $x_n\in\R$ and $\theta_n\in\R$ that converge to zero as $n\to\infty$, and satisfy
\begin{gather*}
    \frac{1}{\|Q\|_{L^2}^2}\int_{\R^d} [e^{i\theta_n}u_n(0,x+x_n)+\mathcal{F}(e^{i\theta_n}u_n(0,x+x_n)-Q,e_-)e_{+}\notag\\
    \quad\quad\quad\quad\quad-\mathcal{F}(e^{i\theta_n}u_n(0,x+x_n)-Q,e_+)e_-]\cdot\overline{iQ}\ dx=0\notag\\
    \frac{1}{\|\partial_jQ\|_{L^2}^2}\int_{\R^d} e^{i\theta_n}u_n(0,x+x_n)\cdot\overline{\partial_jQ}\ dx=0,\quad\ \forall\ j\in\{1,\cdots,d\}.
\end{gather*}
\end{lemma}
\begin{proof}
We define the mapping $J: (u,\theta, \tilde{x})\to(J_0,J_1,\cdots,J_d)$, where
\begin{gather*}
    J_0: (u,\theta,\tilde{x})\rightarrow\frac{1}{\|Q\|_{L^2}^2}\int_{\R^d} [e^{i\theta}u(0,x+\tilde{x})+\mathcal{F}(e^{i\theta}u(0,x+\tilde{x})-Q,e_-)e_{+}\\
    \quad\quad\quad\quad\quad-\mathcal{F}(e^{i\theta}u(0,x+\tilde{x})-Q,e_+)e_-]\cdot\overline{iQ}\ dx\\
    J_j:(u,\theta,\tilde{x})\rightarrow\frac{1}{\|\partial_jQ\|_{L^2}^2}\int_{\R^d} e^{i\theta}u(0,x+\tilde{x})\cdot\overline{\partial_jQ}\ dx\quad\ \forall\ j\in\{1,\cdots,d\}.
\end{gather*}
  One can easily check that $J(Q,0,0)=0$ and that the Jacobian matrix of $J$ is invertible at point $(Q,0,0)$. Therefore, Lemma \ref{analysis1} follows from  \eqref{assumption} and the implicit function theorem.
\end{proof}
Replacing $u_n(t,x)$ by $e^{i\theta}u(t,x+x_n)$, we now also assume that 
\begin{equation}\label{condition2}
\gamma_{n,j}(0)=0\quad \text{ for }\quad  j\in\{0,\cdots,d\}.
\end{equation}

The rest of this section is devoted to  prove the following theorem, 
which plays a key role in the proof of Theorem \ref{mainthm}.  Roughly speaking, the exit times $T_n^{\pm}$ will eventually dominated by  the initial data $\alpha_n^{\mp}(0)$.

\begin{theorem}
\label{byproc}
There exist a (universal)\footnote{In the remaining part of this section, when we refer to ``universal'', it means (the property holds) for any sequence $\{u_n\}$ that satisfies \eqref{assumption}--\eqref{condition2}.  
} constant $\eta_0$, such that for all $\eta\in (0,\eta_0)$,  if $|\alpha_n^-(0)|\geq| \alpha_n^+(0)|,\ \forall\ n\geq 0$,
then for large $n$, 
\begin{align*}
T_n^+(\eta)&=\inf\big\{t\geq 0 \;:\;\,|\alpha_n^-(t)|\geq \eta\big\}
\end{align*}
is finite and
\begin{equation}
\label{estq}
\lim_{n\rightarrow +\infty} \frac{T_n^+(\eta)}{\log |\alpha_n^-(0)|}=\frac{1}{\lambda_1}.
\end{equation}
Furthermore,
\begin{equation}
\label{estqq}
\liminf_{n\rightarrow +\infty} |\frac{d}{dt}\alpha_n^{-}(T_n^+(\eta))|\geq \eta\lambda_1.
\end{equation} 
Similarly, if $|\alpha_n^-(0)|\leq| \alpha_n^+(0)|,\ \forall\ n\geq 0$, 
then for large $n$, 
\begin{align*}
   T_n^-(\eta)&=\sup\big\{t\leq 0 \;:\;\,|\alpha_n^+(t)|\geq \eta\big\} 
\end{align*}
is finite and
\begin{equation}
\label{estq1}
\lim_{n\rightarrow +\infty} \frac{|T_n^-(\eta)|}{\log |\alpha_n^+(0)|}=\frac{1}{\lambda_1}.
\end{equation}
Furthermore,
\begin{equation}
\label{estqq1}
 \liminf_{n\rightarrow +\infty} |\frac{d}{dt}\alpha_n^+(T_n^-(\eta))|\geq {\eta\lambda_1}.
\end{equation} 
\end{theorem}

We start with the following two lemmas, which claim that the growth of $\|h_n\|_{H^1_x}$ will be dominated by $\alpha_n^{-}$($\alpha_n^{+}$) forward(backward) in time direction before time $T_n^{+}$($T_n^{-}$), and $\alpha_n^{-}$($\alpha_n^{+}$) will grow at the exponential speed, which will eventually  lead to the explicit estimate  of the exit times(i.e. \eqref{estq} and \eqref{estq1}). First, let us show that $\|h_n(0)\|_{H^1}$ can be bounded by $\max\{|\alpha_n^{+}(0)|,|\alpha_n^-(0)|\}$.
\begin{lemma}
\label{hnbound}
There exists a (universal) constant $M_0>0$ such that 
\[\alpha_n^+(0)\alpha_n^-(0)\neq 0,\quad \alpha_n^+(0)\alpha_n^-(0)>0\ \text{ and }\ \limsup_{n\rightarrow +\infty} \frac{\|h_n(0)\|_{H^1}+\epsilon_n}{\max\{\left|\alpha_n^+(0)\right|,\left|\alpha_n^-(0)\right|\}}\leq M_0.
\]

\end{lemma}
\begin{proof}
By directly calculation, we get
\[ 
E(Q)+M(Q)-\epsilon_n^2=E(Q+h_n)+M(Q+h_n)=E(Q)+M(Q)+\mathcal{F}(h_n)+O\left(\| h_n\|^3_{H^1}\right).
\]
The expression \eqref{condition1} of $h$ at $t=0$ yields, in view of \eqref{condition2},
\[\| h_n(0)\|_{H^1}\lesssim (|\alpha_n^+(0)|+||\alpha_n^-(0)|+\| g_n(0)\|_{H^1}).
\]
Since the functions $iQ$ and $\partial_j Q$ are in the kernel of $\mathcal{F}$ for $j\in\{1,\cdots,d\}$, we get
\[
\mathcal{F}(h_n(0))=-2\alpha_n^{+}(0)\alpha_n^-(0)+\mathcal{F}(g_n(0)).
\]
Combining the preceding estimates, we obtain
\[
 2\alpha_n^{+}(0)\alpha_n^-(0)=\epsilon_n^2+\mathcal{F}(g_n(0))+O\left(|\alpha_n^{+}(0)|^3+|\alpha_n^{-}(0)|^3+\| g_n(0)\|_{H^1}^3\right).
 \]
By Theorem \ref{positivity}, there exists a positive constant $c_{d,p}$ such that $\mathcal{F}(g_n(0))\geq c_{d,p}\| g_n(0)\|_{H^1}^2$. This yields for large $n$,
\[
2\alpha_n^{+}(0)\alpha_n^-(0) \geq \epsilon_n^2+c_{d,p}\|g_n(0)\|_{H^1}^2\geq \epsilon_n^2+c_{d,p}[c\|h_n(0)\|_{H^1}^2-|\alpha_n^-(0)|-|\alpha_n^+(0)|],
\]
which concludes the proof of the lemma.
\end{proof}
Our next result is the following Lemma:
\begin{lemma}[Growth on $\lbrack -T_n^-,T_n^+ \rbrack$]
\label{key}
Let us fix $\lambda_1^+$ and $\lambda_1^-$, sufficiently close to $\lambda_1$, such that $\lambda_1^-<\lambda_1<\lambda_1^+$.
There exists a positive constant $K_0$ (depending only on  $\lambda_1^{\pm}$) with the following property.  
Let $M> M_0$ (where $M_0$ is given by Lemma \ref{hnbound}). Let $\eta$ such that
\begin{equation}
\label{bootstrapassumption}
0<\eta<\frac{1}{K_0 [M^3+M^2]}.
\end{equation}
We define
\begin{equation}\label{tn+def}
t_n^+=t_n^+(M,\eta)=\inf\big\{t\geq 0 \;:\;\|h_n(t)\|_{H^1_x}\geq M|\alpha_n^-(t)|\text{ or }|\alpha_n^-(t)|\geq \eta\big\}.
\end{equation}
\begin{equation}\label{tn-def}
t_n^-=t_n^-(M,\eta)=\sup\big\{t\leq 0 \;:\;\|h_n(t)\|_{H^1_x}\geq M|\alpha_n^+(t)|\text{ or }|\alpha_n^+(t)|\geq \eta\big\}.
\end{equation}
Then there exists $\tilde{n}>0$ such that for $n\geq \tilde{n}$, if $|\alpha_n^-(0)|\geq|\alpha_n^+(0)|$, then $t_n^+$ is well-defined, positive, and
\begin{align}
\label{estqqq}
\forall\ t\in \left[0,t_n^+\right), &\quad \lambda_1^-|\alpha_n^-(t)|\leq |\frac{d}{dt}\alpha_n^-(t)|\leq \lambda_1^+|\alpha_n^-(t)|\\
\label{estqqqq}
\forall\ t\in \left[0,t_n^+\right), &\quad \frac{1}{K_0}\left|\alpha_n^-(0)\right|e^{\lambda_1^- t}\leq \left|\alpha_n^-(t)\right|\leq K_0 \left|\alpha_n^-(0)\right|e^{\lambda_1^+ t}\\
&\quad\|h_n(t)\|_{H^1_x} \leq 
    K_0|\alpha_n^-(t)|,\quad \forall\ t\in [0,t_n^+)\label{estqqqqq}.
\end{align}
Similarly, if $|\alpha_n^-(0)|\leq|\alpha_n^+(0)|$, then $t_n^-$ is well-defined, negative, and
\begin{align}
\forall t\in \left(t_n^-,0\right], &\quad \lambda_1^-|\alpha_n^+(t)|\leq |\frac{d}{dt}\alpha_n^+(t)|\leq \lambda_1^+|\alpha_n^+(t)|\label{estqqq1}\\
\forall t\in \left(t_n^-,0\right], &\quad \frac{1}{K_0}\left|\alpha_n^+(0)\right|e^{-\lambda_1^+ t}\leq \left|\alpha_n^+(t)\right|\leq K_0 \left|\alpha_n^+(0)\right|e^{-\lambda_1^- t}\label{estqqqq1}\\
&\quad \|h_n(t)\|_{H_x^1}\leq K_0|\alpha_n^+(t)|,\quad \forall\ t\in (t_n^-,0].\label{estqqqqq1}
\end{align}
\end{lemma}
Taking this lemma for granted, let us prove Theorem \ref{byproc}.
\begin{proof}[Proof of Theorem \ref{byproc}]
Without loss of generality, we only prove \eqref{estq} and \eqref{estqq}, assuming $|\alpha_n^-(0)|\geq|\alpha_n^+(0)|$.
By taking $M=1+\max\{M_0,K_0\}$ in Lemma \ref{key}, we obtain that for any $t\in[0,t_n^+)$,
\[
\|h_n(t)\|_{H^1_x}\leq K_0|\alpha_n^-(t)|<M|\alpha_n^-(t)|,
\] 
which further yields
\[
t_n^+(\eta,M)=\inf\big\{t\geq 0 \;:\;\,|\alpha_n^-(t)|\geq \eta\big\}=T_n^+(\eta).
\]
Now by \eqref{estqqqq} and continuity of 
$\alpha_n$, we see that $T_n^+(\eta)\in (0,+\infty)$, $\alpha_n^-(T_n^+(\eta))=\eta$ for all large $n$. Moreover, as $\alpha_n^-(0)\to0$, there must hold $T_n^+(\eta)\to\infty$.

On the other hand, by \eqref{estqqqq}, we readily obtain that for large $n$,
\[
\frac{1}{K_0}\left|\alpha_n^-(0)\right|e^{\lambda_1^-T_n^+(\eta)}\leq \eta \leq K_0 \left|\alpha_n^-(0)\right|e^{\lambda_1^+T_n^+(\eta)},
\]
which is equal to
\[
\log(K_0)+ \lambda_1^- T_n(\eta)\leq\big|\log |\alpha_n^-(0)|\big|+\log \eta \leq \log(K_0)+ \lambda_1^+ T_n(\eta).
\]
As $\alpha_n^-(0)\to0$,
\[
\frac{1}{\lambda_1^+}\leq \liminf_{n\rightarrow +\infty} \frac{T_n(\eta)}{\big|\log |\alpha_n^-(0)|\big|}\leq \limsup_{n\rightarrow +\infty} \frac{T_n(\eta)}{\big|\log |\alpha_n^-(0)|\big|}\leq\frac{1}{\lambda_1^-}.
\]
Since $\lambda_1^{\pm}$ can be chosen arbitrarily near $\lambda_1$, we can conclude the proof of \eqref{estq}.

Finally, by \eqref{estqqq}, \eqref{estqq} naturally holds  and the proof of Theorem \ref{byproc} is complete. \end{proof}

Now we start to  prove  Lemma \ref{key}.

\begin{proof}[Proof of Lemma \ref{key}]
Without loss of generality, we prove all  estimate \eqref{estqqq}-\eqref{estqqqqq} under the assumption $|\alpha_n^-(0)|\geq|\alpha_n^+(0)|$, and the other side estimates can be easily proved by adapting the following argument. As $\alpha_n^{\pm}(0)$ tends to $0$,  the continuity of $\alpha_n^{\pm}(t)$ and $\|h_n(t)\|_{H^1_x}$, for all large $n$, the times $t_n^{\pm}$ exist and are both non-zero. Furthermore,
\begin{align}
\label{qqest1}
\forall\ n\ \text{ and }\ \forall\ t\in [0,t_n^+), &\quad |\alpha_n^-(t)|\leq \eta\\
\label{qqqest1}
\forall\ n\ \text{ and }\ \forall\ t\in [0,t_n^+), &\quad\|h_n(t)\|_{H^1_x}\leq M\left|\alpha_n^-(t)\right|.
\end{align}

We first prove \eqref{estqqq}. Obviously, it is sufficient to prove
\begin{equation}
\label{inequalitydiff}
\left|\frac{d}{dt}\alpha_n^{-}-\lambda_1\alpha_n^{-}\right|\leq 
    m\left|\alpha_n^{-}(t)\right|, \forall\ t\in [0,t_n^+),
\end{equation} 
where $$m=\min\left\{\lambda_1^2-(\lambda_1^-)^2,(\lambda_1^+)^2-\lambda_1^2\right\}.$$
In fact we will prove a stronger estimate:
\begin{align}\label{strongv}
    \left|\frac{d}{dt}\alpha_n^{\pm}\pm\lambda_1\alpha_n^{\pm}\right|\leq 
    m\left|\alpha_n^{-}(t)\right|, \forall\ t\in [0,t_n^+).
\end{align}
Differentiating  the equality $\alpha_n^{\pm}=-\mathcal{F}(h_n,e_{\mp})$, we get, by equation \eqref{linearized},
\begin{align*}
 &\frac{d}{dt}\alpha_n^{\pm}\pm\lambda_1\alpha_n^{\pm}=-\mathcal{F}(\partial_t h_n\,e_{\mp})\mp \lambda_1\mathcal{F}(h_n,e_{\mp})=-\mathcal{F}(\partial_t h_n\pm\lambda_1h_n,e_{\mp})\\
 &=-\mathcal{F}(R(h_n), e_{\mp})-\mathcal{F}\left(Lh_n\pm\lambda_1h_n,e_{\mp}\right)=-\mathcal{F}(R(h_n), e_{\mp}).   
\end{align*}

Thus there exists a constant $C_1$, independent of all parameters, such that
\begin{equation}
\label{diffiential}
\left|\frac{d}{dt}\alpha_n^{\pm}\pm\lambda_1\alpha_n^{\pm}\right|\leq C_1\|h_n\|_{H^1_x}^2.
\end{equation} 
By \eqref{qqest1} and \eqref{qqqest1},
$$ \left|\frac{d}{dt}\alpha_n^{\pm}\pm\lambda_1\alpha_n^{\pm}\right|\leq C_1M^2 \eta |\alpha_n^{-}|,\quad \forall\ t\in[0,t_n^+).$$
Therefore, if $C_1M^2\eta\leq m$ (which follows from \eqref{bootstrapassumption} if $K_0$ is large enough), the desired estimate \eqref{strongv}(and hence \eqref{inequalitydiff}) holds. 

We next prove \eqref{estqqqq}. First we claim that
\begin{align}\label{zzb2}
\frac{d}{dt}\alpha_n^-(t)\cdot\alpha_n^-(t)>0,\quad \alpha_n^-(t)\alpha_n^-(0)> 0,\quad \forall\ t\in[0,t_n^+), 
\end{align}
which implies that $\alpha_n^-$ is strict monotony  and never change the sign on $[0,t_n^+)$. In fact, by using \eqref{inequalitydiff} we see that  $\frac{d}{dt}\alpha_n^{-}(0)\cdot\alpha_n^-(0)>0$. Therefore, \eqref{zzb2} must hold on some small interval $[0,t_0)$, which  implies that $\alpha_n^-(t_0)\neq0$, and then by \eqref{inequalitydiff}, also implies $\frac{d}{dt}\alpha_n^{-}( t_0)\cdot\alpha_n^-(t_0)>0$. Now the claim \eqref{zzb2} holds by the bootstrap argument.

Combining \eqref{inequalitydiff} and \eqref{zzb2}, we readily obtain
   \[
   \lambda_1^-\alpha_n^{-}(t)\leq \frac{d}{dt} \alpha_n^{-}(t)\leq \lambda_1^+\alpha_n^{-}(t), \forall\ t\in[0,t_n^+),\quad \text{ if } \alpha_n^{-}(0)>0,
   \]
\[
   \lambda_1^+\alpha_n^{-}(t)\leq \frac{d}{dt} \alpha_n^{-}(t)\leq \lambda_1^-\alpha_n^{-}(t), \forall\ t\in[0,t_n^+),\quad \text{ if } \alpha_n^{-}(0)<0,
   \]
which yields \eqref{estqqqq} by invoking $\lambda_1^-|\alpha_n^{-}(0)|\leq |\frac{d}{dt}\alpha_n^{-}(0)|\leq \lambda_1^+|\alpha_n^{-}(0)|$.

Finally, we turn to prove \eqref{estqqqqq}.
Noticing that by \eqref{estqqq} and \eqref{strongv}, the following differential inequality  holds:
   \begin{equation}
      \left|\frac{d}{dt}(e^{\lambda_1t}\alpha_n^{+})\right|= \left|\frac{d}{dt}\alpha_n^{+}+\lambda_1\alpha_n^{+}\right|\leq 
    \left|\frac{d}{dt}\alpha_n^{-}(t)\right|, \forall\ t\in [0,t_n^+),
   \end{equation}   
if we choose $\lambda_1^-,\lambda_1^+$ sufficiently close to $\lambda_1$. Since  $\alpha_n^{-}$ is strict monotony on the interval $[0,t_n^+)$, we readily obtain
\begin{equation}\label{harm2}
    |\alpha_n^{+}(t)|\leq e^{\lambda_1t}|\alpha_n^{+}(t)|<|\alpha_n^-(t)|,\quad \forall\ t\in  [0,t_n^+).
\end{equation}

Next, we show that there exists a constant $C_1>0$, independent of the parameters $M$ and $\eta$, such that for all  $t\in \left[0,t_n^+\right]$,
\begin{gather}
\label{estimate_energy2}
\left\|g_n\right\|_{H^1_x}+\epsilon_n\leq C_1|\alpha_n^-|+C_1\|h_n\|_{H^1_x}^{3/2}\\
\label{estimate_compactness2}
\|h_n\|_{H^1_x}\leq C_1\left(\left|\alpha_n^-\right|+\sum_{j=0}^d \left|\gamma_{j,n}\right|\right).
\end{gather}
In fact, noticing that
\[
E\left(Q+h_n\right)+M(Q+h_n)=E(Q)+M(Q)-\epsilon_n^2
\]
Thus there exists a constant $\tilde{C}>0$ (independent of the parameters) such that
\begin{equation}\label{qest1}
    \left|\mathcal{F}((h_n)+\epsilon_n^2\right|\leq \tilde{C}\|h_n(t)\|_{H^1_x}^{3}.
\end{equation}
Furthermore, by \eqref{condition1} and Lemma \ref{positivity}, we have
    \begin{gather*}
        \mathcal{F}(h_n)=-2\alpha_n^+(t)\alpha_n^{-}(t)+\mathcal{F}\left(g_n\right)\\
        \mathcal{F}\left(g_n\right)\sim \left\| g_n\right\|_{H^1_x}^2
    \end{gather*}
Inserting these estimates into \eqref{qest1} and then using  \eqref{harm2}, it is easy to check that the desired estimate  \eqref{estimate_energy2} holds.

Let us show \eqref{estimate_compactness2}. 
By  \eqref{harm2} and \eqref{estimate_energy2}, for  any $t\in[0,t_n^+)$,
\begin{equation*}
    \left\| h_n\right\|_{H^1_x}\leq C\left[\left|\alpha_n^-\right|+\sum_{j=0}^d \left|\gamma_{j,n}\right|+\left\| g_n\right\|_{H^1_x}\right]\leq C \left[\left|\alpha_n^-\right|+\sum_{j=0}^d \left|\gamma_{j,n}\right|+\|h_n\|_{H^1_x}^{\frac 32}\right]
\end{equation*}
for some constant $C>0$. Recalling  \eqref{qqqest1}, we readily obtain
\begin{equation*}
\left\|h_n\right\|_{H^1_x}\leq 
    C\left[\left|\alpha_n^-\right|+\sum_{j=0}^d \left|\gamma_{j,n}\right|\right]+C\|h_n\|_{H^1_x}M^{1/2}\eta^{1/2}, \ \forall \ t\in[0,t_n^+),
\end{equation*}
then by \eqref{bootstrapassumption}, we obtain \eqref{estimate_compactness2}.

With estimates \eqref{estimate_energy2} and \eqref{estimate_compactness2} in hand, we proceed to  prove \eqref{estqqqqq}. For this, we claim that there exists a (universal) constant $C_2$  such that
\begin{equation}
\label{diffgamma}
\forall\ j\in \{0,\cdots,d\},\; \left|\gamma_{j,n}(t)\right|\leq
    C_2 \left|\alpha_n^-(t)\right|,\ \forall\ t\in[0,t_n^+).
\end{equation} 
Take this claim for granted, let us complete the proof of \eqref{estqqqqq}. Inserting \eqref{diffgamma} into  \eqref{estimate_compactness2}, we see that if we choose $K_0=C_1C_2$(which only depends on $\lambda_1$ if $|\lambda_1^{\pm}-\lambda_1|$ sufficiently small), then the desired estimate \eqref{estqqqqq} naturally holds.

Now let us prove \eqref{diffgamma}. As each $\partial_j Q$ is in the kernel of $\mathcal{L}$, we have
\begin{align*}
\gamma_{0,n}'(t)=\Re\int_{\R^d}  &\partial_th_n(t)\overline{iQ}\ dx=-\Im\int_{\R^d}  R(h_n)Q\ dx-\Im\int_{\R^d}  \mathcal{L}h_n\cdot Q\ dx\\
&=-\Im\int_{\R^d}  R(h_n)Q\ dx-\Im\int_{\R^d} \mathcal{L}(\alpha_n^{+}e_{+}+\alpha_n^{-}e_-+g_n)Q\ dx,    
\end{align*}
and for $j=1\ldots d$,
\begin{align*}
   \gamma_{j,n}'(t)=\Re\int_{\R^d} & \partial_th_n(t)\partial_j Q\ dx=\Re\int_{\R^d}  R(h_n)\partial_j Q\ dx+\Re\int_{\R^d}  \mathcal{L}h_n\cdot \partial_j Q\ dx\\
   &=\Re\int_{\R^d}  R(h_n)\partial_j Q\ dx+\Re\int_{\R^d} \mathcal{L}(\alpha_n^{+}e_{+}+\alpha_n^{-}e_-+g_n)\partial_j Q\ dx. 
\end{align*}
Therefore, by \eqref{estqqq},  \eqref{qqest1} and \eqref{qqqest1}, there exists a (universal) constant $C_3$ such that
\begin{align}
\label{diffgammaj}
\left|\gamma_{j,n}'(t)\right|&\leq C_3 \left(|\alpha^{-}(t)|+\| h_n(t)\|^{3/2}_{H^1}+\| h_n(t)\|^2_{H^1}\right)\notag\\
&\leq C_3 \left(|\alpha^{-}(t)|+[\| h_n(t)\|^{1/2}_{H^1}+\| h_n(t)\|_{H^1}]\| h_n(t)\|_{H^1}\right)\notag\\
&\leq C_3
    \frac{1+M^{3/2}\eta^{1/2}+M^2\eta}{\lambda_1^-}|\frac{d}{dt}\alpha_n^-(t)|,\ \forall\ t\in[0,t_n^+).
\end{align} 

Integrating between $0$ and $t$, and using that $\gamma_{j,n}(0)=0$, and the fact that  $\alpha_n^{-}$ is strict monotony on the interval $[0,t_n^+)$, we can complete the proof of \eqref{diffgamma}.
\end{proof}

\section{Proof of main result}
In this section we complete the proof Theorem \ref{mainthm}. We start with the following lemma, which is based on the classification results below and at the threshold.
\begin{lemma}\label{compactness}
Let $\{u_n\}$ be a sequence of solutions to \eqref{1.1} that satisfies
\begin{equation}\label{aaz}
  M(u_n)=M(Q),\quad  \|u_n\|_{\dot{H}^1}<\|Q\|_{\dot{H}^1},\quad E(u_n)=E(Q)-\epsilon_n^2,\quad \lim_{n\to\infty}\epsilon_n\to0,
\end{equation}
\begin{equation}\label{zzb1}
\left\|u_{n}\right\|_{S(-\infty,0)}=\left\|u_{n}\right\|_{S(0,+\infty)}\underset{n\rightarrow +\infty}{\longrightarrow} +\infty.
\end{equation}
Then there exist sequences $x_n\in\R$ and $\theta_n\in\R$ such that by passing to a subsequence, we have
\[
\lim_{n\to\infty}\|e^{i\theta_n}u_n(0,x+x_n)-Q\|_{H^1_x}=0.
\]
\end{lemma}
\begin{proof}[Sketch of the proof]
    We first prove that  by passing to a subsequence, there exists $v_0\in H^1$ and sequences $\tilde{x}_n\in\R$ and $\tilde{\theta}_n\in\R$ such that
    \begin{equation}\label{convergence1}
         \lim_{n\to\infty}\|e^{i\theta_n}u_n(0,x+x_n)-v_0\|_{H^1_x}=0.
    \end{equation}
    The argument to prove this is based on the concentration/compactness argument, which  is now quite standard. So we only give a sketch. By  the linear $H^1$ profile decomposition  to $u_n(0)$,  we can  construct  \emph{nonlinear profiles} and then use Theorem \ref{thm:claener} to argue that there must be exactly one profile, which implies \eqref{convergence1}.

   Now we find that to conclude the proof, it suffices to prove that $v_0=e^{i\theta_0}Q(x+x_0)$ for some $x_0\in\R$ and $\theta_0\in\R$. From \eqref{aaz} and \eqref{convergence1}, we obtain
   \[
   M(v_0)=M(Q),\quad E(v_0)=E(Q),\quad \|v_0\|_{\dot{H}^1}\leq\|Q\|_{\dot{H}^1}
   \]
Moreover, Let $v$ be the maximal-lifespan solution to \eqref{1.1} with $v|_{t=0}=v_0$,  by the local theory, \eqref{zzb1} also yields
\[
    \|v\|_{S(-\infty,0]}=\|v\|_{S[0,\infty)}=+\infty.
\]
   Therefore, by Theorem \ref{thm:claener2}, up to symmetries(not include scaling), $v$ must coincide the $e^{it}Q$, which  directly yields $v_0=e^{i\theta_0}Q(x+x_0)$ for some $x_0\in\R$ and $\theta_0\in\R$. 
\end{proof}

We also need the following lemma:
\begin{lemma}\label{smallness}
    There exist constants $c_0>0$, $C_0>0$ and $0<s_0\leq 1 $ such that for any interval $I$, $0\in I$, $|I|<c_0$ and any sequence $\{u_n\}$ that satisfies \eqref{assumption}--\eqref{condition2}, we have
   \begin{equation}\label{mmmn}
            \|Q-u_n\|_{S(I)}\leq C_0\|u_n(0)-Q\|_{\dot{H}^s}^{s_0}   
   \end{equation}
\end{lemma}
\begin{proof}
   Since $\|Q\|_{X(I)}^{q_1}=|I|^{q_1}\|Q\|_{L_x^{r_1}}$, we first choose $c_0\leq c$, where $c$ is defined in Lemma \ref{perpo}, then we have
   \[
   \|h\|_{X(I)}\leq C\|u_n(0)-Q\|_{\dot{H}^{s
   _c}},
   \]
   which, by invoking standard Strichartz estimates, means that
   $\||\nabla|^{s_c}u_n\|_{\dot{S}^0}$ is uniformly bounded (provided $c_0$ sufficiently small). Now \eqref{mmmn} can be obtained from the  interpolation between $X(I)$ and $\||\nabla|^{s_c}u_n\|_{\dot{S}^0}$.
\end{proof}

With the above lemmas in hand, we are ready to prove our main theorem:
\begin{proof}[Proof of Theorem \ref{mainthm}]
 We divide the proof into two steps. We first show the lower bound estimate.

\textbf{Step 1. Lower bound }
We first show
\begin{equation}
\label{part1}
\liminf_{\epsilon\rightarrow 0^+} \frac{\mathcal{I}_{\epsilon}}{\left|\log \epsilon\right|}\geq \frac{2}{\lambda_1}\int_{\R^d} Q^{\frac{(p-1)(d+2)}{2}}\ dx.
\end{equation} 
We argue by contradiction. If \eqref{part1} does not hold, there exists a sequence $\epsilon_n$ which tends to $0$ such that for some $\lambda_0>\lambda_1$
\begin{equation}
\label{xxvc}
\forall n,\quad \frac{2}{\lambda_0}\int_{\R^d} Q^{\frac{(p-1)(d+2)}{2}}\ dx\geq \frac{\mathcal{I}_{\epsilon_n}}{\left|\log \epsilon_n\right|}.
\end{equation} 
Since we have assumed that $\Re\int_{\R^d} \nabla Q \cdot \overline{\nabla e_{\pm}}\ dx=\int_{\R^d} \nabla Q\cdot\nabla \Re e_{+}\ dx>0$,  we can choose positive sequences $a_n$, $b_n$ that converge to zero, and satisfy
\begin{equation*}
   u_{n,0}={(1-b_n)Q-a_n e_+-a_ne_-},\quad \|u_{n,0}\|_{L^2}=\|Q\|_{L^2},\quad \|\nabla u_{n,0}\|_{L^2}<\|\nabla Q\|_{L^2}
\end{equation*}
\begin{equation}
\label{llj}
 \epsilon_n^2=M(Q)+E(Q)-M(u_{n,0})-E\left(u_{n,0}\right),\quad  \quad \epsilon_n^2 \sim b_n\sim a_n^2
\end{equation} 
as $n\to +\infty$. This is possible because for any small $a,b>0$, we have the following identities:
\begin{align*}
    M((1-b)Q-a e_+-ae_-)&=M(Q)+(b^2-2b)M(Q)-2a\Re\int_{\R^d} Q\overline{e_{+}}\ dx+4a^2\int_{\R^d} (\Re e_+)^2\ dx\\
    &=M(Q)+(b^2-2b)M(Q)+4a^2\int_{\R^d} (\Re e_+)^2\ dx
\end{align*}
\begin{align*}
    \|\nabla [(1-b)Q-a e_+-ae_-]\|^2_{L^2}&=\|\nabla Q\|^2_{L^2}+(b^2-2b)\|\nabla Q\|^2_{L^2}\\
    &-2a\Re\int_{\R^d} \nabla Q\cdot \overline{\nabla e_{+}}\ dx+4a^2\int_{\R^d} (\Re e_+)^2\ dx
\end{align*}
\begin{align*}
    &E\left((1-b)Q-a e_+-ae_-\right)+M((1-b)Q-a e_+-ae_-)\\
    =&E(Q)+M(Q)+\mathcal{F}(-bQ-a e_+-ae_-)\\
    =&E(Q)+M(Q)+2a^2\mathcal{F}(e_+,e_-)+b^2\mathcal{F}(Q)+2ab\mathcal{F}(Q,e_+)\\
    =&E(Q)+M(Q)-2a^2+b^2\mathcal{F}(Q)+2ab\mathcal{F}(Q,e_+),
\end{align*}
where the first identity comes from the asymptotic decay property of $Q^{\pm}$(See \cite{campos}):
\[
M(Q^{\pm})=M(Q),\quad \|Q^{\pm}-e^{it}Q\mp e^{it-\lambda_1t}e_+\|_{L^2}\lesssim e^{-2\lambda_1t}\Rightarrow \Re \int Q \overline{e_+}=0.
\]

Now we consider the family of solutions $\{u_n\}_n$ of \eqref{1.1} with initial data $u_{n,0}$. Following the analysis below the Lemma \ref{analysis1},  we can also set $\gamma_{n,j}(0)=0$ for $j\in\{0,\cdots,d\}$. Then by Theorem \ref{byproc}, noting that $|\alpha_n^+(0)|=|\alpha_n^-(0)|\sim |a_n|$, we readily obtain
\begin{equation*}
\lim_{n\rightarrow +\infty} \frac{T_n^+(\eta)}{\left|\log a_n\right|}=\lim_{n\rightarrow +\infty} \frac{T_n^-(\eta)}{\big|\log a_n\big|}=\frac{1}{\lambda_1} 
\end{equation*}
By \eqref{llj},
\begin{equation}
\label{rrt}
\lim_{n\rightarrow +\infty} \frac{T_n^-(\eta)}{\big|\log |\epsilon_n|\big|}=\lim_{n\rightarrow +\infty} \frac{T_n^+(\eta)}{\big|\log |\epsilon_n|\big|}=\frac{1}{\lambda_1}.
\end{equation}
Next we decompose $e^{-it}u_n=Q+h_n$, by triangle inequality,
\[
\|u_n\|_{S(0,T_n^+(\eta))}\geq \|Q\|_{S(0,T_n^+(\eta))}-\|h_n\|_{S(0,T_n^+(\eta))}.
\] 
Obviously,
\[
\|Q\|_{S(0,T_n^+(\eta))}=  (T_n^{+}(\eta))^{\frac{2}{(p-1)(d+2)}}\|Q\|_{L^{\frac{(p-1)(d+2)}{2}}}.
\]
So we only need to consider the contribution of $\|h_n\|_{S(0,T_n^+(\eta))}$(provided $\eta$ sufficiently small). As  $\| h_n\|_{H^1_x}\leq K_0\eta$ on $(0,T_n^+(\eta))$,  we can  split the interval $(0,T_n^+(\eta))$ into finitely many subintervals at length of $c_0$ given by Lemma \ref{smallness}), provided $\eta$ sufficiently small. Then, by applying Lemma \ref{smallness}, we see that $h_n$ has a scattering norm with the size of at most $ C_0K_0\eta$ on each such subinterval, where $C_0$ is an universal constant, and as a consequence, $h_n$ satisfies the following estimate on the entire interval $(0, T_n^+(\eta))$: 
\[
\|h_n\|_{S(0,T_n^+(\eta))}^{\frac{(p-1)(d+2)}{2}}\lesssim\left(\frac{T_n^+(\eta)}{c_0}+1\right)\eta^{\frac{(p-1)(d+2)}{2}s_0}\lesssim \eta^{\frac{(p-1)(d+2)s_0}{2}}T_n^+(\eta).
\]
Combining the preceding estimates, we readily obtain
\[ \int_0^{T_n^+(\eta)} \int_{\R^d}|u_n|^{\frac{(p-1)(d+2)}{2}} \ dxdt\geq T_n^{+}(\eta)\left[\|Q\|_{L^\frac{(p-1)(d+2)}{2}}-C\eta^{s_0}\right]^{\frac{(p-1)(d+2)}{2}}
\]
for an universal constant $C$. Hence with \eqref{rrt},
\begin{equation*}
\liminf_{n\rightarrow +\infty}\frac{1}{|\log \epsilon_n|}\int_0^{+\infty} \int_{\R^d}|u_n|^{\frac{(p-1)(d+2)}{2}} \ dxdt
\geq \frac{1}{\lambda_1}\left[\|Q\|_{L^\frac{(p-1)(d+2)}{2}}-C\eta^{s_0}\right]^{\frac{(p-1)(d+2)}{2}}.
\end{equation*} 
Letting $\eta$ tends to $0$ we derive 
\begin{equation}\label{llm1}
    \liminf_{n\rightarrow+\infty} \frac{1}{|\log \epsilon_n |} \int_0^{+\infty} |u_n|^{\frac{(p-1)(d+2)}{2}}\ dxdt\geq \frac{1}{\lambda_1}\|Q\|_{L^\frac{(p-1)(d+2)}{2}}^{\frac{(p-1)(d+2)}{2}}.
\end{equation}
Repeating the above argument, one can also prove the analogous estimate for the negative time direction:
\begin{equation}\label{llm2}
    \liminf_{n\rightarrow+\infty} \frac{1}{|\log \epsilon_n |} \int_{-\infty}^{0} |u_n|^{\frac{(p-1)(d+2)}{2}}\ dxdt\geq \frac{1}{\lambda_1}\|Q\|_{L^\frac{(p-1)(d+2)}{2}}^{\frac{(p-1)(d+2)}{2}}.
\end{equation}
Now  \eqref{llm1} and \eqref{llm2} contradict \eqref{xxvc}.

\textbf{Step 2. Upper bound}

We next show the upper bound on $\mathcal{I}_{\epsilon}$, i.e that 
\begin{equation}
\label{part2}
\limsup_{\epsilon\rightarrow 0^+} \frac{\mathcal{I}_{\epsilon}}{\left|\log \epsilon\right|}\leq \frac{2}{\lambda_1}\int_{\R^d} Q^{\frac{(p-1)(d+2)}{2}}.
\end{equation}

For this we will show that if $\epsilon_n>0$ is a sequence that goes to $0$ and $\{u_n\}$ a sequence of solutions of \eqref{1.1} such that
\begin{equation*}
\|u_{n}(0)\|_{L^2}^{(1-s_c)/s_c}\| u_{n}(0)\|_{\dot{H}^1}<\|Q\|_{L^2}^{(1-s_c)/s_c}\| Q\|_{\dot{H}^1},
\end{equation*} 
\begin{equation*}
    M(u_n)^{(1-s_c)/s_c}E(u_n)= M(Q)^{(1-s_c)/s_c}(E(Q)-\epsilon^2),
\end{equation*}
then 
\begin{equation}
\label{CV}
\limsup_{n\rightarrow +\infty} \frac{1}{\left|\log \epsilon_n\right|}\int_{\R\times\R^d} |u_n|^{\frac{(p-1)(d+2)}{2}}\ dxdt\leq \frac{2}{\lambda_1}\int_{\R^d} Q^{\frac{(p-1)(d+2)}{2}}\ dx.
\end{equation}
After rescaling and time-translating,  we may assume
\begin{equation*}
  M(u_n)=M(Q),\quad  \|u_n\|_{\dot{H}^1}<\|Q\|_{\dot{H}^1},\quad E(u_n)=E(Q)-\epsilon_n^2, \label{final20}
\end{equation*}
\begin{equation}
\label{final2}
\left\|u_{n}\right\|_{S(-\infty,0)}=\left\|u_{n}\right\|_{S(0,+\infty)}\underset{n\rightarrow +\infty}{\longrightarrow} +\infty.
\end{equation}
Furthermore, by Lemma \ref{compactness} and the analysis below Lemma \ref{analysis1}, after  translating $u_n$ in space and preforming a phase rotation if necessary, we can  assume
\[
\lim_{n\rightarrow+\infty} \|u_n(0)-Q\|_{H^1_x}=0,\quad \gamma_{n,j}(0)=0,\quad \forall\ j\in\{0,\cdots,d\}.
\] 

Fix a small $\eta>0$, and consider $T_n^+(\eta), T_n^-(\eta)$ defined by Theorem \ref{byproc}. We first show that there exists a constant $C>0$ such that
\begin{equation*}
\limsup_{n\rightarrow +\infty} \frac{1}{|\log \epsilon_n|}\int_{T_n^-(\eta)}^{T_n^+(\eta)} \int_{\R^d} |u_n|^{\frac{(p-1)(d+2)}{2}}\ dxdt\leq \frac{2}{\lambda_1}\left[\|Q\|_{L^{\frac{(p-1)(d+2)}{2}}}+C\eta^{s_0}\right]^{\frac{(p-1)(d+2)}{2}}.
\end{equation*}
We only devote to prove
\begin{equation}
\label{final3}
\limsup_{n\rightarrow +\infty} \frac{1}{|\log \epsilon_n|}\int_{0}^{T_n^+(\eta)} \int_{\R^d} |u_n|^{\frac{(p-1)(d+2)}{2}}\ dxdt\leq \frac{1}{\lambda_1}\left[\|Q\|_{L^{\frac{(p-1)(d+2)}{2}}}+C\eta^{s_0}\right]^{\frac{(p-1)(d+2)}{2}},
\end{equation}
since the proof of the analogy of this estimate for the negative time is completely similar.

We decompose $u_n$ as
\[
e^{-it}u_n(t)=Q+h_n=Q+\alpha_n^+(t)e_++\alpha_n^-(t)e_-+\gamma_{n,0}(t)iQ+\sum_{j=1}^d\gamma_{n,j}\partial_jQ+g_n,\quad g_n\in\mathcal{B}^{\bot}.
\]
Noticing that there exists $\theta_n\in\R$ so that $e^{i\theta_n}u_n(0)=\overline{u_n(0)}$, replacing $u_n$ by $e^{i\theta_n}u_n$  if necessary, we can also assume
\begin{equation*}
|\alpha_n^-(0)|\geq |\alpha_n^+(0)|.
\end{equation*} 

First, one may easily obtain
\begin{equation}\label{xxc}
    \int_0^{T_n^+(\eta)} \int_{\R^d}|u_n|^{\frac{(p-1)(d+2)}{2}}\ dxdt \leq  T_n^+(\eta)\left[\|Q\|_{L^{\frac{(p-1)(d+2)}{2}}}+C\eta^{s_0}\right]^{\frac{(p-1)(d+2)}{2}}
\end{equation}
by adopting the argument in Step 1.  Therefore,
\[
\limsup_{n\rightarrow+\infty} \frac{1}{T_n^+(\eta)}\int_0^{T_n^+(\eta)} \int_{\R^d}|u_n|^{\frac{(p-1)(d+2)}{2}} \ dxdt\leq
\left[\|Q\|_{L^{\frac{(p-1)(d+2)}{2}}}+C\eta^{s_0}\right]^{\frac{(p-1)(d+2)}{2}}
\]
On the other hand,  Lemma \ref{hnbound} yields 
\begin{equation*} 
\epsilon_n\leq M_0|\alpha_n^-(0)|
\end{equation*} 
for any sufficiently large $n$. Therefore, by Theorem \ref{byproc}, 
\begin{equation}
\label{estimate_eps}
\limsup_{n\rightarrow+\infty} \frac{T_n^+(\eta)}{|\log \epsilon_n|}\leq \frac{1}{\lambda_1},
\end{equation}
which together with \eqref{xxc} implies \eqref{final3}. 

Finally, we claim that if $\eta$ is small enough, there exists a constant $C(\eta)>0$ such that for large $n$
\begin{equation}
\label{final4} 
\|u_n\|_{S(-\infty,T_n^-(\eta))}+\|u_n\|_{S(T_n^+(\eta),+\infty)}\leq C(\eta), 
\end{equation}
which is enough to prove \eqref{part2} and hence conclude the proof.
To see this,  we observe that \eqref{final3} and and \eqref{final4},
\begin{align*}
\limsup_{n\rightarrow+\infty} \frac{1}{|\log \epsilon_n|}\|u_n\|_{S(0,+\infty)}^{\frac{(p-1)(d+2)}{2}}&=\limsup_{n\rightarrow+\infty} \frac{1}{|\log \epsilon_n|}\|u_n\|_{S(0,T_n^+(\eta))}^{\frac{(p-1)(d+2)}{2}}\\
&\leq \frac{1}{\lambda_1}\left[\|Q\|_{L^{\frac{(p-1)(d+2)}{2}}}+C\eta^{s_0}\right]^{\frac{(p-1)(d+2)}{2}}
\end{align*}
and
\begin{align*}
\limsup_{n\rightarrow+\infty} \frac{1}{|\log \epsilon_n|}\|u_n\|_{S(-\infty,0)}^{\frac{(p-1)(d+2)}{2}}&=\limsup_{n\rightarrow+\infty} \frac{1}{|\log \epsilon_n|}\|u_n\|_{S(-\infty,T_n^-(\eta))}^{\frac{(p-1)(d+2)}{2}}\\
&\leq \frac{1}{\lambda_1}\left[\|Q\|_{L^{\frac{(p-1)(d+2)}{2}}}+C\eta^{s_0}\right]^{\frac{(p-1)(d+2)}{2}}.
\end{align*}
Letting $\eta\to0$, we readily obtain
\[
\limsup_{n\rightarrow+\infty} \frac{1}{|\log \epsilon_n|}\|u_n\|_{S(0,+\infty)}^{\frac{(p-1)(d+2)}{2}}\leq \frac{2}{\lambda_1}\|Q\|_{L^{\frac{(p-1)(d+2)}{2}}}^{\frac{(p-1)(d+2)}{2}},
\]
which shows, in view of \eqref{final2}, the desired estimate \eqref{CV}.

It therefore remains to prove \eqref{final4}. We will argue by contradiction.  Without loss of generality, we may assume that 
\[
\limsup_{n\to+\infty}\|u_n\|_{S(T_n^+(\eta),+\infty)}=+\infty.
\]
By passing to the subsequence, 
\begin{equation}
\label{final5}
\|u_n\|_{S(T_{n}^+(\eta),+\infty)} \to +\infty \quad \text{ as }\quad n\to+\infty.
\end{equation}
Furthermore, by \eqref{final2} 
\begin{equation}
\label{final6} 
\|u_{n}\|_{S(-\infty,T_n^+(\eta))}\to +\infty \quad \text{ as }\quad n\to+\infty.
\end{equation} 
In view of \eqref{final5} and \eqref{final6}, Lemma \ref{compactness}  implies that there exist parameters $\theta_n\in \R$ and $x_n\in \R^d$ such that(up to the subsequence) 
\begin{equation*}
\lim_{n\rightarrow +\infty} \left\| e^{i\theta_n}u_n\left(T_{n}^+(\eta),x-x_n\right)-Q\right\|_{H_x^1}=0.
\end{equation*}
Since for each $n$, $e^{it}\tilde{Q}=e^{-i\theta_n+t}Q(x+x_n)$ is a solution of \eqref{1.1},  we have
\[
i\partial_tu(T_n^+(\eta))=-\Delta[u-\tilde{Q}](T_n^+(\eta))-[|u|^{p-1}u-|\tilde{Q}|^{p-1}\tilde{Q}](T_n^+(\eta))-\tilde{Q}.
\]
Now by the Sobolev embedding,  we arrive at
\begin{align*}
    &\|\partial_t u_n(T_n^+(\eta))-iu_n(T_n^+(\eta))\|_{H_x^{-1}}\leq\|\partial_t u_n(T_n^+(\eta))-i\tilde{Q}\|_{H_x^{-1}}+\| u_n(T_n^+(\eta))-\tilde{Q}\|_{H_x^{-1}}\\
    &\leq \| u_n(T_n^+(\eta))-\tilde{Q}\|_{H_x^{-1}}+\|\Delta [u_n-\tilde{Q}](T_n^+(\eta))\|_{H_x^{-1}}+\|[|u_n|^{p-1}u_n-|\tilde{Q}|^{p-1}\tilde{Q}](T_n^+(\eta))\|_{L_x^{\frac{(p-1)(d+2)}{(p-1)(d+2) - 2}}}\\
    &\leq  \| [u_n-\tilde{Q}](T_n^+(\eta))\|_{H_x^{1}}+\|u_n(T_n^+(\eta))-\tilde{Q}\|_{H_x^1}[\|u_n\|^{p-1}_{H^1_x}+\|Q\|^{p-1}_{H^1_x}]\\
    &\to 0
\end{align*}
as $n\to\infty$. Since $e_+$ is Schwartz function, we get
\[
\mathcal{F}(e^{-iT_n^+(\eta)}[\partial_tu_n(T_n^+(\eta))-iu(T_n^+(\eta))],e_+)\to 0
\]
 and hence
\begin{align*}
\frac{d}{dt}\alpha_n^-(T_n^+(\eta))=-\frac{d}{dt}\mathcal{F}(e^{-it}u_n-Q,e_+)(T_n^+(\eta))&=-\mathcal{F}(e^{-iT_n^+(\eta)}[\partial_tu_n(T_n^+(\eta))-iu(T_n^+(\eta))],e_+)\\
&\to 0.
\end{align*}
This is a contradiction, since Theorem \ref{byproc} yields
\[
\liminf_{n\to \infty} |\alpha_n^-(T_n^+(\eta))|\geq \lambda_1\eta.
\]
Combining \eqref{part1} and \eqref{part2}, we can conclude the proof.
\end{proof}

\end{document}